\documentclass[12pt]{amsart}  
\usepackage{amssymb}
\usepackage{amscd}
\usepackage{xypic}

\subjclass[2000]{18F15, 53D15}
\keywords{Symplectic manifold, Lagrangian submanifold, deformation quantization}

\newtheorem{thm}{Theorem}[section]
\newtheorem{cor}[thm]{Corollary}
\newtheorem{prop}[thm]{Proposition}
\newtheorem{lemma}[thm]{Lemma}

\theoremstyle{remark}
\newtheorem{remark}[thm]{Remark}
\newtheorem{example}[thm]{Example}

\theoremstyle{definition}
\newtheorem{definition}[thm]{Definition}

\numberwithin{equation}{subsection}

\newcommand{\A}{{\widehat{\mathbb A}}}
\newcommand{\V}{{\mathcal V}}
\newcommand{\caH}{{\mathcal H}}
\newcommand{\caE}{{\mathcal E}}
\newcommand{\E}{{\mathcal E}}
\newcommand{\caP}{{\mathcal P}}
\newcommand{\caL}{{\mathcal L}}
\newcommand{\bV}{{\mathbb V}}
\newcommand{\C}{{\mathbb C}}
\newcommand{\fC}{\widehat{\mathbb C}}

\newcommand{\Z}{{\mathbb Z}}
\newcommand{\R}{{\mathbb R}}
\newcommand{\oiho}{\frac{1}{i\hbar}\omega}
\newcommand{\oih}{\frac{1}{i\hbar}}

\newcommand{\tG}{\widetilde{G}}
\newcommand{\tP}{\widetilde{P}}

\newcommand{\fx}{{\widehat{x}}}
\newcommand{\fy}{{\widehat{y}}}
\newcommand{\fz}{{\widehat{z}}}
\newcommand{\fxi}{{\widehat{\xi}}}
\newcommand{\feta}{{\widehat{\eta}}}

\newcommand{\isomoto}{\overset{\sim}{\to}}
\newcommand{\lisomoto}{\overset{\sim}{\leftarrow}}

\newcommand{\ab}{_{\alpha\beta}}
\newcommand{\abc}{_{\alpha\beta\gamma}}
\newcommand{\hor}{\operatorname{hor}}
\newcommand{\can}{\operatorname{can}}

\newcommand{\Lie}{\operatorname{Lie}}
\newcommand{\Sp}{\operatorname{Sp}}
\newcommand{\tSp}{\widetilde{\operatorname{Sp}}}

\newcommand{\T}{\mathbf{T}}

\newcommand{\g}{\mathfrak{g}}
\newcommand{\p}{\mathfrak{p}}

\newcommand{\jets}{\operatorname{jets}}
\newcommand{\Aut}{\operatorname{Aut}}

\newcommand{\Isom}{\operatorname{Isom}}

\newcommand{\tnabla}{\widetilde{\nabla}}

\newcommand{\ML}{\operatorname{ML}}
\newcommand{\OL}{\operatorname{OL}}
\newcommand{\GL}{\operatorname{GL}}
\newcommand{\Der}{\operatorname{Der}}
\newcommand{\cont}{\operatorname{cont}}

\newcommand{\ddx}{\frac{\partial}{\partial x}}
\newcommand{\ddfx}{\frac{\partial}{\partial \fx}}
\newcommand{\ddxi}{\frac{\partial}{\partial \xi}}
\newcommand{\ddfxi}{\frac{\partial}{\partial \fxi}}
\newcommand{\fxh}{\frac{\fx}{i\hbar}}

\newcommand{\xh}{\frac{x}{i\hbar}}

\newcommand{\DR}{\mathtt{DR}}

\newcommand{\Ad}{\operatorname{Ad}}

\newcommand{\tg}{{\widetilde{\mathfrak g}}}
\newcommand{\tp}{{\widetilde{\mathfrak p}}}

\newcommand{\caHIJ}{\caH^{\fy_{I,J}}}

\makeatletter
\renewcommand{\subsubsection}{\@startsection
{subsubsection}%
{2}%
{0mm}%
{-\baselineskip}%
{-0.5\baselineskip}%
{\normalfont\normalsize\bfseries }}%
\makeatother

\begin{document}

\title{Oscillatory modules}

\author[B.Tsygan]{Boris Tsygan}
\address{Department of
Mathematics, Northwestern University, Evanston, IL 60208-2730, USA}
\email{tsygan@math.northwestern.edu}

\begin{abstract}
Developing the ideas of Bressler and Soibelman and of Karabegov, we introduce a notion of an oscillatory module on a symplectic manifold which is a sheaf of modules over the sheaf of deformation quantization algebras with an additional structure. We compare the category of oscillatory modules on a torus to the Fukaya category as computed by Polishchuk and Zaslow.
\end{abstract}

\thanks{Partially supported by NSF grant DMS-0605030}
\maketitle \vspace{-1cm}

\bigskip

{\centerline{\em{In memory of Moshe Flato}}}

\smallskip

\section{Introduction}
Deformation quantization was first introduced in \cite{bffls}. For a symplectic manifold $(M,\, \omega)$, a deformation quantization algebra is $C^\infty(M)[[\hbar]]$ with a new associative, $\C[[\hbar]]$-bilinear product
$$f*g=fg+\sum_{k=1}^\infty (i\hbar)^kP_k(f,g)$$
 where $1*f=f*1=f$, $P_k$ are bidifferential operators, and $\oih(f*g-g*f)\equiv \{f,g\}(\rm{mod}\; \hbar),\;$ $\{f,g\}$ being the Poisson bracket associated to $\omega$. Since \cite{bffls} was motivated both by quantum mechanics and by representation theory, quantities of the type $\exp{\oih}S(x,\xi)$ played a prominent role there, as did modules over the deformed algebra. This is true with regard to many subsequent works on deformation quantization, in particular \cite{DJ} and \cite{BdM}.

Asymptotic analysis of quantities of the type $\exp({\oih}S(x,\xi))$ at $\hbar\to 0$ plays a crucial role in linear PDEs and in particular in quantum mechanics, cf. for example \cite{GS}, \cite{H}, \cite{L}, \cite{M}, \cite{MSS}, \cite{NSS}. This seems to have motivated not only \cite{bffls} but also Fedosov's approach to deformation quantization \cite{Fedosov}, as well as other works on the subject. An idea to systematically put the theory from Guillemin and Sternberg's book \cite{GS} into Fedosov's context was suggested to the author by A. Karabegov. We started to realize it in \cite{NT}. The present paper is a continuation of this work.

Let $(M,\, \omega)$ be a symplectic manifold. A deformation quantization algebra of $M$ can be viewed, according to \cite{Fedosov}, as the algebra of horizontal sections of a bundle of algebras ${\A}_M$ with a product-preserving flat connection $\nabla _{\mathbb A}.$ Locally in Darboux coordinates $(x,\xi)=(x_1,\ldots,x_n,\xi_1,\ldots, \xi_n)$, sections of $\A_M$ are functions $f(x,\xi,\fx,\fxi, \hbar)$ where $\fx=(\fx_1,\ldots, \fx_n)$ and $\fxi=(\fxi_1,\ldots, \fxi_n)$ are formal variables; the product is the $\Sp$-equivariant Moyal-Weyl product in $\fx$ and $\fxi$; the connection is, after a gauge transformation, equal to
$$\nabla_{\mathbb A}=(\ddx-\ddfx)dx+(\ddxi-\ddfxi)d\xi$$
The sheaf of $\nabla_{\mathbb A}$-horizontal sections of $\A_M$ is isomorphic to $C^\infty_M[[\hbar]]$ (locally, horizontal sections are of the form $f(x+\fx, \xi+\fxi)$ where $f$ is a smooth function). The product makes this sheaf a deformation quantization of $(M, \omega).$ In particular, a bundle $V$ of ${\A}_M$-modules with a compatible flat connection $\nabla _V$ gives rise to a sheaf of modules over the deformed algebra.

In the approach that is presented in this paper, expressions $\exp(\oih S)$ are viewed as sections of a groupoid on $M$. We start with the local, or rather formal, situation. Let $\A=\C[[\fx,\fxi,\hbar]]$ with the Moyal-Weyl product. Its (continuous) derivations are all of the form $[\oih H, ?]$ for $H\in \A.$ They form a Lie algebra $\g$ that admits a central extension $\tg=\oih \A$ with the bracket $[f,g]=f*g-g*f.$ Let $G$ be the group of continuous automorphisms of $\A$ that induce a real symplectic isomorphism of $\C^{2n}.$ We define a central extension $\tG$ of $G$ (cf. \ref{ss:Deformation quantization of a formal neighborhood}). Elements of $\tG$ can be written as $\exp(\oih S(\fx,\fxi, \hbar))$ where $S$ is a power series such that $S({\rm{mod}}\hbar)$ is a series in $\fx,\,\fxi$ without a linear part and with real quadratic part. Note that a constant part is allowed. In other words, the construction of $\tG$ out of $G$ includes formally allowing expressions $\exp(\oih a+b)$ where $a$ and $b$ are complex numbers. Note also that $\Sp(2n, \R)$ is a subgroup of $G$, and the universal cover $\tSp(2n, \R)$ is a subgroup of $\tG$.

Now we have a group $\tG$, a Lie algebra $\tg$, and an algebra $\A$; $\tG$ acts on $\tg$ and $\A$, and $\tg$ acts on $\A$. All these actions are compatible. Next (in \ref{ss:The algebraic Weil representation}) we construct a space $\caH$ on which $\tG$, $\tg$, and $\A$ all act in a compatible way. The space is an algebraic analog of the Weil metaplectic representation. Elements of $\caH$ are expressions of the type $\exp(\oih S(\fx, \hbar))f(\fx, \hbar)$ and their (partial) Fourier transforms, as explained in \ref{ss:The algebraic Weil representation}. Here $S$ is a power series as described above; $f$ belongs to a certain completion of the space of power series.

Now consider a symplectic manifold $(M, \omega).$ The Fedosov bundle of algebras $\A_M$ is a bundle with the fiber $\A.$ There are also a bundle of Lie algebras $\tg_M$ with the fiber $\tg$ and a bundle of groups with the fiber $\tG$. We extend this bundle of groups to a groupoid on $M$ (in \ref{ss:The groupoid GM}).

For $x,\,y$ in $M$, let $G_{x,y}$ be the set of (continuous) isomorphisms of algebras $\A _y\isomoto \A _x.$ The sets $G_{x,y}$ form a groupoid on $M$ that we denote by $G_M$. An element $g\in G_{x,y}$ defines an isomorphism $g_0: T_yM\isomoto T_xM$ (the linear part of $g$). Assume that $c_1(M)=0.$ We define a new groupoid $\tilde{G}_M$ which is an extension of $G_M$. For any $x\in M,$ $(\tG_M)_{x,x}$ (or simply $\tG_{x,x}$) is isomorphic to an extension of the fundamental group $\pi _1(M, x)$ by  $\tG$.

For example, if $M=\R^{2n},$ consider the trivial vector bundle with the fiber $\caH$. The set $\tG_{(x_1,\xi_1),(x_2,\xi_2)}$ consists of isomorphisms $\caH_{(x_2,\xi_2)}\isomoto \caH_{(x_1,\xi_1)}$ that are in the image of the group $\tG$. One can write sections of $\tG$ on $M\times M$ as expressions $\exp(\oih S)$ where $S=S(x_1,x_2, \xi_1, \xi_2, \fx, \fxi, \hbar)$ such that the linear part of $S({\rm{mod}}\hbar)$ with respect to $\fx$ and $\fxi$ is equal to
$$S_{\rm{lin}}=S_ {\fx}\fx+S_\fxi \fxi=(\xi_1-\xi_2)\fx-(x_1-x_2)\fxi$$
and the quadratic part
$$S_{\rm{quad}}=\frac{1}{2}S_{\fx\fx}\fx^2+S_{\fx\fxi}\fx\fxi+\frac{1}{2}S_{\fxi\fxi}\fxi^2$$
is real. Let $S_{\rm{nl}}=S-S_{\rm{lin}}.$ We understand $\exp(\oih S)$ as the product
\begin{equation}\label{eq:Slin, nl}\exp(\oih S)=\exp(\oih S_{\rm{lin}})\exp(\oih S_{\rm{nl }})
\end{equation}
where $\exp(\oih S_{\rm{nl }})$ is viewed as an element of $\tG_{(x_2, \xi_2), (x_2, \xi_2)}$ and $\exp(\oih S_{\rm{lin}})$ as a special element of $\tG_{(x_1, \xi_1), (x_2, \xi_2)}$, a parallel transport $\caH_{(x_2,\xi_2)}\isomoto \caH_{(x_1,\xi_1)}$.

\begin{remark}\label{rmk:meaning of S} Note that the interpretation of the terms $\exp(\oih S_{\rm{lin}})$ that is proposed here requires some explanation. For a linear function $l$ on $\R^{2n},$ we denote by $\exp(\oih l(\fx, \fxi))$ the parallel transport by the vector $v$ which is the image of $l$ under the isomorphism $\R^{2n, *}\isomoto \R^{2n}$ defined by the symplectic form. Parallel transports by different vectors {\em commute} with each other, and so do elements $\exp(\oih l(\fx, \fxi))$ of the groupoid $\tG$. A natural impulse would be to make them satisfy commutation relations similar to the ones in the Heisenberg group. There are indeed sections of the groupoid $\tG$ that satisfy these relations, and they play an important role, in particular in \ref{sss:The twisted bundle HM on torus}. (Among other features, they have better behavior with respect to the flat connection on $\tG$)
. For example, one can put $g_1(x,\xi)=\exp(\oih \fxi);$ $g_2=\exp(\oih x)\exp(\oih \fx)$. The first is the same parallel transport in the $x$ direction at every point; the second is the parallel transport in the $\xi$ direction, but followed by multiplication by the element $\exp(\oih x)$ of the group $\tG$ at the point $(x,\xi).$

This choice of notation works well for us, in particuar when we consider the case of the torus. But it has its disadvantages. For example, it is a little counterintuitive that elements $\exp(\oih \fx)$ and $\exp (\oih \fxi)$ commute, as do $\exp(\oih \fx)$ and $\exp( \frac{1}{2i\hbar}\fxi ^2).$ Perhaps it would be more consistent to denote parallel translations by $\exp(\oih l({\underline \fx}, {\underline \fxi}))$ and to distinguish ${\underline \fx}, {\underline \fxi}$ from $\fx, \fxi.$ On the other hand, our notation is convenient to describe the summands of \eqref{eq:theta intro} below.
\end{remark}
An {\em oscillatory module} on $M$ is a bundle $\V$ with a flat connection $\nabla_\V$ on which $\A_M$, $\tg_M$, and $\tG_M$ act compatibly with each other and with the connections. In other words, $\V$ is a bundle of $\A_M[\hbar^{-1}]$-modules together with an isomorphism $T_g: \V_y \isomoto \V_x$ for any $g\in \tilde{G}_{x,y}$, and a flat connection $\nabla_{\V}$, subject to various compatibilities. Intuitively, an oscillatory module may be viewed as a module over the ``algebra'' of expressions $\sum \oih S(x,\xi,\hbar))a(x,\xi,\hbar).$

Another intuitive way to understand oscillatory modules is via the framework of Connes' noncommutative geometry. It is well understood within this framework that the role of the ring of functions on a quotient of a space by a group is played by the cross product of the group by the algebra of functions on the space \cite{C}, \cite{CM}. In case of a good action, this algebra is Morita equivalent to the algebra of functions invariant under the action of the group, that is of functions on the quotient space. Therefore one can say that an oscillatory module $V$ is a module over the deformation quantization of the noncommutative geometric quotient of $M$ by all (quantized) symplectomorphisms.

It is natural to try to construct an oscillatory module globalizing the construction of the algebraic Weil representation $\caH$ above. Indeed, it is easy to construct a bundle of $\A_M$-modules with a fiber $\caH$ and a connection; however, the curvature of the connection is not zero but $\oih \omega$ (for a specific choice of $(\A_M, \nabla_{\mathbb A})$). This is a direct consequence of the Fedosov theory. What we can do is to pass to a {\em twisted bundle} $\caH_M$ (in \ref{ss:The twisted bundle HM}). By definition, a twisted bundle is given by a \v{C}ech one-cochain ${\bf g}_{\alpha\beta}:U_\alpha\cap U_{\beta}\to \tG$ satisfying
$${\bf g}_{\alpha\beta}{\bf g}_{\beta\gamma}= \exp(\oih c_{\alpha\beta\gamma}){\bf g}_{\alpha\gamma}$$
where $c_{\alpha\beta\gamma}$ is a \v{C}ech two-cocycle with constant coefficients representing the class of the symplectic form $\omega$.

Note that this definition is very similar to the definition of a flat gerbe (\cite{Brylinski}); for a relation between gerbes and the Fedosov construction, cf. \cite{Deligne}.

For twisted bundles, two notions from the theory of usual vector bundles still make sense. One is the notion of a flat connection (because the factor $\exp(\oih c_{\alpha\beta\gamma})$, being locally constant, preserves any connection). The other is the groupoid of homotopy classes $(x_t\in M; \;g_t:\caH_{x_0}\lisomoto \caH_{x_t};\;x_0=x;\;x_1=y; \; g_0={\rm{id}}).$ (Note that neither an individual fiber $\caH_x$ nor the set of isomorphisms $\caH_x\lisomoto \caH_y$ makes any sense). Actually, we {\em define} our groupoid $\tG_M$ this way, in terms of $\caH_M.$

The twisted bundle satisfies all the axioms for an oscillatory module, except being an actual vector bundle. To construct oscillatory modules from it, we use a cut-and-paste procedure starting with a Lagrangian submanifold (cf. \ref{ss:Oscillatory modules and Lagrangian submanifolds}).

For a Lagrangian submanifold $L$, note that $\omega|L$ is exact by definition. Therefore we can choose a primitive one-cochain for a \v{C}ech cocycle representing $\omega|L$ and make $\caH_M|L$ a vector bundle with a flat connection. We define a groupoid $\tP_L$ together with its morphism to $\tG_M|L.$ This groupoid consists of classes of $(x_t, g_t)$ such that the linear part of $g_t$ preserves the tangent space to $L$.  We also define a subbundle $\bV_L$ of $\caH_M$ that is stable under $\A_M|L,$ $\tg_M|L,$ and $\tP_L.$

\begin{remark} \label{rmk:NT and Hormander} In \cite{NT}, we studied an asymptotic version of the ${\mathcal D}$-module of Lagrangian distributions whose wave front is a given Lagrangian submanifold of the cotangent bundle $M=T^*X$. The bundle $\bV_L$ is a direct generalization of this construction.
\end{remark}

To assign to a Lagrangian submanifold $L$ an oscillatory module $\V_L,$ one can use the induction procedure familiar from representation theory and define
$$\V_{L}=\tG_M\times _{\tP_L}\bV _L$$
or, more precisely,
$$(\V_{L})_{x}=\{\sum _{y\in L,\, g_{x,y}\in \tG_{x,y},\, v_y\in (\bV_L)_y}g_{x,y}v_y\}/\sim$$
where $(g_{x,y}p_{y,z})v_z\sim g_{x,y}(p_{y,z}v_z)$ for $p_{y,z}\in \tP_{y,z}.$ The sums we are considering are finite, or rather, after some simple completion, countable.

\begin{remark}\label{rmk:Gin, et al} This construction passes from a ${\mathcal D}$-module-like object ${\mathbb V}_L$ to an induced module $\V_L$ on which quantized Hamiltonian symplectomorhisms $\exp(\oih S)$ act. This reminds of a construction that passes from a ${\mathcal D}$-module on $\{f\neq 0\}$ to a new module on which expressions $f^s$ act (cf. \cite{Gin}). It would be interesting to compare the two objects, as well as to compare our work to \cite{Kara} and related works. It would also be useful to compare our groupoid $\tG_M$ to quantized symplectomorphisms that appear, for example, in \cite{FG}, \cite{KS}, and \cite{Ta}, as well as to the star exponentials from \cite{BdM} and \cite{DJ}. Another question that we intend to study is what happens when the condition $c_1(TM)=0$ is removed, and how to relate our constructions to the works \cite{BdMG} and \cite{Guil}.
\end{remark}

Let us illustrate the construction of the module $\V_L$ with an example. Put $M={\mathbb T}^2$ with periodic coordinates $x,\xi$ and the standard symplectic structure $\omega=d\xi dx.$

Recall that the Fukaya category of a compact symplectic manifold with $c_1(M)=0$ is an $A_\infty$ category whose objects are Lagrangian submanifolds with some extra data, cf. \cite{Fu}, \cite{FOOO}. The homotopy category of this $A_\infty$ category is a category with the same objects. In the case of the standard two-torus, this category was computed by Polishchuk and Zaslow in \cite{PZ}. They showed that it is equivalent to the derived category of coherent sheaves on the mirror dual complex torus. This is a partial case of the homological mirror symmetry conjecture of Kontsevich \cite{KHMS}.

Let $L_m$ be the Lagrangian submanifold $\xi=mx.$ Sections of the oscillatory module $\V_{L_m}$ are expressions $\sum_{k\in \Z} \exp((\xi-mx-k)\fx)f_k(x,\xi)$ where $f_k$ are two-periodic functions with values in $\caH.$ We will show that morphisms of oscillatory modules $\V_{L_m}\to \V_{L_{m+n}},\, n>0,$ compose similarly to morphisms in the Fukaya category of $M$, via Polishchuk and Zaslow's mirror symmetry. In fact, we construct such morphisms as operators of multiplication by
\begin{equation}\label{eq:theta intro}\sum_{\nu\equiv a\;{\operatorname{mod}}\frac{1}{n}\Z}e^{-\frac{n}{2i\hbar}(x+\fx+\nu)^2}
\end{equation}
Here
$$e^{-\frac{n}{2i\hbar}(x+\fx+\nu)^2}=e^{-\frac{n}{i\hbar}(x+\nu)\fx}
e^{-\frac{n}{2i\hbar}\fx^2
-\frac{n}{2i\hbar}(x+\nu)^2}$$
is an element of $\tG$ that is understood as in \eqref{eq:Slin, nl} above.
The expressions $\theta _a$ are, formally, linear combinations of theta functions of level $n$, cf. \cite{Theta}, with $\tau=\hbar$ and $z=x+\fx,$ after an application of the Poisson summation formula and the substitution $\tau\mapsto -\frac{1}{\tau}.$ Therefore, their composition satisfies the same identities as multiplication of theta functions and as composition in the Fukaya category of $\T^2$ (because of the mirror symmetry for the torus, cf. \cite{PZ}). Note that D. Tamarkin has a version of his sheaf-theoretic category on the torus that coincides with the Fukaya category via Polishchuk-Zaslow \cite{Ta2}.

From this example it looks like the notion of an oscillatory module is closer to the Fukaya category than the usual notion of a module over a deformed algebra; it may be also helpful for studying the B-side. This, together with our ongoing work on the index theorem for Fourier integral operators \cite{LNT} and a general project to study deformations and modules over them in relation to index theory (\cite{bgnt1}, \cite{bgnt2}, \cite{bgnt3}, \cite{bgnt4}, \cite{DAP}, \cite{KS}, \cite{KS1}, \cite{KS2}, ), is the motivation behind the present work.

\begin{remark}\label{rmk:history}
The idea to relate Lagrangian intersections to deformation quantization was communicated to the author by Boris Feigin in mid eighties, before either the invention of the Fukaya category or the more recent advances in deformation theory. The first work on a relation between $D$-modules and the Fukaya theory is Bressler and Soibelman's article \cite{BS}. This relation is advanced in \cite{KW} with applications to a physical interpretation of the geometric Langlands program. A rigorous expression of the Fukaya theory of $T^*X$ in terms of the category of constructible sheaves on $X$ is contained in \cite{NZ}. Further work in this direction was done in \cite{FW}, \cite{GW}, \cite{N}, \cite{Zaslow et al 1}, \cite{Zaslow et al 2}. In \cite{Ta}, another version of the sheaf-theoretical category on $T^*X$ is proposed, together with some hints about how such a category could be glued together from Darboux charts to a symplectic manifold. This category is related to the Fukaya category; it is used to prove non-displaceability results for some Lagrangian submanifolds and other compact subsets. Hopefully, oscillatory modules will provide a partner for Tamarkin's sheaf-theoretical objects under some sort of a Riemann-Hilbert correspondence. Or, perhaps, there is a theory for complex manifolds that is locally based on neither sheaves nor ${\mathbb A}$-modules but on elliptic pairs in the spirit of \cite{SS}.
\end{remark}
{\bf Aknowledgements} I have greatly benefited from regular discussions with A. Getmanenko and D. Tamarkin. Another great benefit was an opportunity to present a very crude early version of this work at several meetings of the joint Northwestern-UChicago seminar. I am grateful to all participants, in particular to A. Beilinson, R. Bezrukavnikov, D. Ben-Zvi, V. Drinfeld, D. Gaitsgory, E. Getzler, D. Nadler, K. Vilonen and E. Zaslow. I am grateful to R. Bezrukavnikov, P. Bressler, K. Fukaya, A. Goncharov, A. Gorokhovsky, D. Kaledin, M. Kashiwara, D. Kazhdan, R. Nest, Y. Soibelman, A. Szenes, and especially B. Feigin and A. Karabegov
for fruitful discussions and suggestions. I would like to thank the referee for a careful reeding of the manuscript and for many useful suggestions.

\section{Deformation quantization of a formal neighborhood}
\subsection{Deformation quantization algebra and its symmetries}\label{ss:Deformation quantization of a formal neighborhood} Put
$${{{\A}}}={\mathbb C}[[\fx_1, \ldots, \fx_n, \fxi_1, \ldots, \fxi_n, \hbar]]$$
with the Moyal-Weyl product;
$${\mathfrak g}=\Der_{\cont}(\A)=\frac{1}{i\hbar}\A/\frac{1}{i\hbar}{{\mathbb{C}}[[\hbar]]};\;\;{\widetilde{\mathfrak g}}=\frac{1}{i\hbar}\A$$
Put $H={\rm{Sp}}(2n);$ ${\mathfrak h}={\mathfrak{sp}}(2n);$ introduce the grading
\begin{equation}\label{eq:grading}
|\fx_i|=|\fxi_i|=1;\; |\hbar|=2.
 \end{equation}
 We can identify ${\mathfrak g}_0$ with ${\mathfrak{sp}}(2n, \C).$ One has a central extension
\begin{equation}\label{eq:extension g}
0\to\oih\C[[\hbar]]\to\tg\to\g\to 0,
\end{equation}
as well as
\begin{equation}\label{eq:grading g}
\g=\prod_{i=-1}^\infty \g_i;\;\tg=\prod_{i=-2}^\infty \tg_i.
\end{equation}
Any continuous automorphism $g$ of $\A$ induces a symplectic linear transformation $g_0$ of $\C^{2n}.$ Denote by $G$ the group of those $g$ whose linear part $g_0$ preserves the real structure. We have
\begin{equation}\label{eq:G semi dir}
G=\Sp(2n,\R)\ltimes \exp (\g_{\geq 1})
\end{equation}
Define the central extension
\begin{equation}\label{eq:tG semi dir}
\tG=\exp(\oih \C\oplus \C)\times \tSp(2n,\R)\ltimes \exp (\tg_{\geq 1})
\end{equation}
One has an exact sequence
\begin{equation}\label{eq:extension G}
1\to \Z\times \exp(\oih\C[[\hbar]])\to\tG\to G\to 1
\end{equation}
\begin{remark} \label{rmk:Lie algebra vs Lie group}
Note that the Lie algebra of $G$ is not $\g$ but the Lie subalgebra $\g_{\geq 0}.$ The Lie algebra of $\tG$ is not $\tg$ but rather $\tg_{-2}+\tg_{\geq 0}.$ This should be kept in mind throughout the paper. For example, we will consider $\g$-valued or $\tg$-valued connections in $G$- or $\tG$-bundles. Later, when we globalize these constructions on a symplectic manifold $M$, the bundle of Lie algebras $\tg_M$ will play the role of a Lie algebra of the groupoid $\tG_M.$
\end{remark}
Define also $P$ to be the subgroup of $G$ consisting of elements $g$ whose linear part preserves the Lagrangian subspace
\begin{equation}\label{eq:L0}
L_0=\{\fxi_1=\ldots=\fxi_n=0\}
\end{equation}
Let $\tP$ be the preimage of $P$ in $\tG$.
\subsection{The algebraic Weil representation}\label{ss:The algebraic Weil representation} Let $\fy=(\fy_1,\ldots,\fy_n)$ be $n$ formal variables. For a symmetric real $n\times n$ matrix $a$, put
\begin{equation}\label{eq:Ha}
\caH_a^{\fy}=\exp(\frac{a\fy^2}{2i\hbar})\fC[[\fy, \hbar]]((e^{\frac{c}{i\hbar}}|c\in \C))
\end{equation}
Here
\begin{equation}
\fC[[\fy, \hbar]]=\{\sum_{k=-N}^\infty v_k|v_k\in \C[[\fy]]((\hbar))_k\}
\end{equation}
with respect to the grading \eqref{eq:grading}; for any vector space $V,$ we define
\begin{equation} \label{eq:completion}
V((e^{\frac{c}{i\hbar}}|c\in \C))=\{\sum_{k\in {\mathbb N};{\rm{Re}}(c_k)\to+\infty}e^{\frac{c_k}{i\hbar}}v_k\},
\end{equation}
$v_k\in V.$
In particular, the multiplication by $h$ is automatically invertible.

Note that the completion as in \eqref{eq:completion} is not necessary for ${\mathbb A},$ $\tg,$ and $\tG$ to act on $\caH$. But without it, expressions as in \eqref{eq:theta function 2} would not exist.

Put also
\begin{equation}
\caH^\fy=\oplus_a \caH^\fy_a
\end{equation}
For a nondegenerate $a$, define the Fourier transform (cf. \cite{K})
\begin{equation}\label{eq:Fourier}
F: \caH_a^{\fy}\isomoto\caH_{-a^{-1}}^\feta
\end{equation}
as follows. Heuristically,
\begin{equation}
(Ff)(\feta)=\frac {e^{-\frac{\pi i n}{4}}}{(2\pi i \hbar)^{n/2}}
\int e^\frac{\fy\feta}{i\hbar}f(\fy)d\fy;
\end{equation}
To give the above formula a rigorous meaning, put
$$
F(f(\fy)\exp(\frac{a\fy^2}{2i\hbar}))(\feta)
=f(i\hbar\frac{\partial}{\partial\feta})F(\exp(\frac{a\feta^2}{2i\hbar}))=
$$
$$
f(i\hbar\frac{\partial}{\partial\feta})\frac{e^{-\frac{\pi i n}{4}}}{\det(\sqrt{ia})}
\exp(\frac{-a^{-1}\feta^2}{2i\hbar})
=\frac{e^{-\frac{\pi in}{2}}(-1)^{p(a)}}{\det\sqrt{|a|}}f(i\hbar\frac{\partial}{\partial\feta})\exp(\frac{-a^{-1}\feta^2}{2i\hbar})
$$
Here $p(a)$ is the number of positive eigenvalues of $a$. We used the branch of the square root for which $\sqrt(ix)>0$ if $ix>0$; it is defined defined on the complex plane with the line $\{ix<0, \; x\in \R\}$ removed. The final term in the above chain of equalities can be viewed as the definition of the first term.

The definition of the Fourier transform $F$ extends to elements of the form \begin{equation}
{\bf f}(\fy)=\exp(\frac{a\fy^2}{2\pi i\hbar}+i\fy\fz)f(\fy)
\end{equation}
where $a$ is nondegenerate and $\fz$ is another formal parameter:
\begin{equation}\label{eq:Fourier with a shift}
F({\bf f})(\feta)=F(\exp(\frac{a\fy^2}{2i \hbar})f(\fy))(\feta+\fz)
\end{equation}

One has
\begin{equation}\label{eq:Fourier props}
F^2=-1;\;F\fy F^{-1}=i\hbar\frac{\partial}{\partial\feta};\;Fi\hbar\frac{\partial}{\partial\fy}F^{-1}=-\feta
\end{equation}

Now define
\begin{equation}\label{eq:alg Weil}
\caH=\sum_{I\coprod J=\{1,\ldots, n\}}\caH^{\fy_{I,J}}/\sim
\end{equation}
Here $\fy_{I,J}=(\fx_I, \fxi_J).$ The equivalence relation is defined as follows. Let
\begin{equation}\label{eq:element of H}
{\bf f}(\fx_{I_1}, \fx_{I_2}, \fxi_J)=\exp({\frac{S(\fx_{I_1}, \fx_{I_2}, \fxi_J)}{2i\hbar}})f(\fx_{I_1}, \fx_{I_2}, \fxi_J)\in \caH^{\fy_{I_1\coprod I_2, J}}
\end{equation}
 Here $S$ is a quadratic form. We assume it becomes a nondegenerate form in $\fx_{I_2}$ if we put $\fx_{I_1}=\fxi_J=0.$ We declare the following two elements equivalent:
 \begin{equation}\label{eq:element of H 1}
{\bf f}(\fx_{I_1}, \fxi_{I_2}, \fxi_J)=\exp({\frac{S(\fx_{I_1}, \fxi_{I_2}, \fxi_J)}{2i\hbar}})f(\fx_{I_1}, \fxi_{I_2}, \fxi_J)\in \caH^{\fy_{I_1, I_2\coprod  J}}
\end{equation}
and
\begin{equation}
F_{I_2}{\bf f}(\fx_{I_1}, \fxi_{I_2}, \fxi_J)\in \caH^{\fy_{I_1\coprod I_2, J}}
\end{equation}
Here $F_{I_2}$ is the Fourier transform from functions of $\fx_{I_2}$ to functions in $\fx_{I_2}$, as in \eqref{eq:Fourier with a shift}. Our equivalence relation is the minimal one for which all the pairs as above are equivalent.
\subsubsection{The action of $\A$ on $\caH$}\label{sss:The action of A on H} The algebra $\A$ acts on the space $\caH$ as follows. Recall that
$$\fy_{I,J}=(\fx_k|k\in I; \; \fxi_j |j\in J).$$
On the summand $\caH^{\fy_{I,J}}$, $\fx_k$ acts by multiplication if $k\in I$ and by $-i\hbar\frac{\partial}{\partial \fxi_k}$ otherwise; $\fxi_j$ acts by multiplication if $j\in J$ and by $i\hbar\frac{\partial}{\partial \fx_j}$ otherwise. It is clear that this definition preserves the equivalence relation.
\subsubsection{The action of $\tg$ on $\caH$}\label{sss:The action of g on H} The Lie algebra $\tg$ acts on $\caH$ exactly as above, namely, $\tg=\oih \A,$ and the action of $\hbar$ on $\caH$ is invertible.
\subsubsection{The action of $\tG$ on $\caH$}\label{sss:The action of G on H} The subgroup $\exp(\oih \C\oplus \C)$ acts on $\caH$ in the obvious way; the subgroup $\exp (\tg_{\geq 1}),$ via the action of $\tg$ defined above. Indeed, for any $X\in \tg_{\geq 1},$ the exponential of the operator $h\mapsto Xh$ converges as an operator on $\caH$. Finally, $\tSp(2n,\R)$ acts exactly by the same formula as the usual metaplectic representation:
\begin{equation}\label{eq:Weil rep} \left [ \begin{array}{cc}
1&0\\
a&1\end{array}\right ] \mapsto \exp(-\frac{a\fx^2}{2i\hbar});
\;
\left [ \begin{array}{cc}
0&1\\
-1&0\end{array}\right ] \mapsto F;
\end{equation}
\begin{equation}\label{eq:Weil rep 1}
\left [ \begin{array}{cc}
b&0\\
0&^tb^{-1}\end{array}\right ] \mapsto T_b;\;(T_bf)(x)=\frac{1}{\sqrt{\det(b)}}f(b^{-1}x)
\end{equation}
The meaning of the above operators is as follows. The Fourier transform $F$ sends ${\bf f}(\fx_I, \fxi_J)\in \caH^{\fy_{I,J}}$ to ${\bf f}(\fxi_I, -\fx_J)\in \caH^{J,I}.$ The operator $\exp(-\frac{a\fx^2}{2i\hbar})$ extends from $\caH^\fx$ to all of $\caH$ as follows:
$$
\exp(-\frac{a\fx^2}{2i\hbar})=\exp( \oih a_{II}\fx_I^2-2a_{IJ}\fx_I\frac{\partial}{\partial \fxi_J}+i\hbar a_{JJ}(\frac{\partial}{\partial \fxi_J})^2)
$$
on $\caH^{\fy_{I,J}}.$ This is clearly a well-defined operator. Similarly, If $b$ is in the connected component of the identity, then
\[ b=\exp
\left [ \begin{array}{cc}
c_{II} & c_{IJ}\\
^tc_{JI} & c_{JJ}
\end{array} \right ]
\]
Put $a\circ b=\frac{1}{2}(ab+ba).$ Define
$$T_b=\exp(\frac{1}{2}(c_{II}\fx_I\circ \frac{\partial}{\partial \fx_I}+\oih c_{IJ}\fx_I \fxi_J-i\hbar \frac{\partial}{\partial \fx_I}\frac{\partial}{\partial \fxi_J}+c_{JJ}\fxi_J \circ\frac{\partial}{\partial \fxi_J})$$
on $\caH^{\fy_{I,J}}.$ This, too, is a well-defined operator. It remains to define $T_b$ for one $b$ that is not in the connected component of the identity, namely a diagonal matrix with one entry equal to $-1$ and the rest equal to $1$. Depending on whether the entry $-1$ is in the $I$ or $J$ group, we define $T_b$ exactly as in \eqref{eq:Weil rep 1} above, with respect to the variables $\fx_I$ or to $\fxi_J.$
\begin{remark}\label{rmk:orbit}
The construction of $\caH$ mimics very closely the construction of the orbit of $1$ under the action of $\tSp(2n)\ltimes \C[\fx, \fxi]$ on the space of distributions via differential operators and the standard metaplectic representation. The vector ${\bf 1}_{I,J}$ in $\caH^{\fy_{I,J}}$ stands for the partial Fourier transform $F_I(1).$
\end{remark}
\begin{lemma}\label{lemma:weil and lagrangian}
Assign to $\exp({\frac{ a\fy_{I,J}^2}{2i\hbar}})f(\fy_{I,J})\in \caH^{\fy_{I,J}}$
the Lagrangian subspace
$$
\fxi_I=a_{II}\fx_I+a_{IJ}\fxi_J
$$
$$
\fx_J=-^ta_{IJ}\fx_I-a_{JJ}\fxi_J
$$
This is a well-defined map $\coprod\caHIJ\to \Lambda (n)$ where $\Lambda(n)$ is the Grassmannian of Lagrangian subspaces in $\R^{2n}.$ The space $\caH$ is identified with the space of finitely supported sections of a $\tG$-equivariant vector bundle on $\Lambda(n).$
\end{lemma}
\begin{proof} The Lagrangian Grassmannian is a homogeneous space of $\tG$ via the projection $\tG\to \tSp\to \Sp.$ In fact,
$$\Lambda(n)\isomoto \tG/\tP.$$
The above formulas turn the manifold $\Lambda(n)$ into the union of charts indexed by $I, \,J$:
 \begin{equation} \label{eq:Gr via charts}
 \Lambda(n)=\coprod _{I,J} U_{I,J}/\sim
 \end{equation}
 where each chart $U_{I,J}$ is the space of real symmetric $n\times n$ matrices. One checks directly that the equivalence from the definition of $\caH$ intertwines with the equivalence in \eqref{eq:Gr via charts}, and the action of $\tG$ on $\caH$ intertwines with the action of $\tG$ on $\Lambda(n)$ via the projection $\tG\to \Sp$. The latter statement is enough to check on generators of $\tSp$ as in formulas \eqref{eq:Weil rep}, \eqref{eq:Weil rep 1}.
\end{proof}

\begin{lemma}
The lines $\C \cdot {\bf 1}_{I,J}\subset \caHIJ$ form a line subbundle of $\caH$ which is isomorphic to the bundle on $\Lambda(n)$ determined by the \v{C}ech one-cohomology class $(-1)^{\mu _L}$ where $\mu_L$ is the generator of $H^1(\Lambda(n), \Z)$ (the Maslov class).
\end{lemma}
\subsubsection{$(\tG, \tg, \A)$-modules}\label{sss:(G,g,A)-modules}
\begin{definition}\label{dfn:(G,g,A)-modules}
A $(\tG, \tg, \A)$ is a vector space $V$ equipped with an action of the group $\tG$, of the Lie algebra $\tg,$ and of the associative algebra $\A$ subject to
$$ g(Xv) = g(X)gv; g(av) = g(a)gv; X(av)=X(a)v+a(Xv)$$
for $g\in \tG, \; X\in \tg, \; a\in \A.$
\end{definition}
\begin{lemma} \label{lemma:H as a(G,g,A)-module}
The actions described in \ref{sss:The action of G on H}, \ref{sss:The action of G on H}, and \ref{sss:The action of G on H} turn $\caH$ into a $(\tG, \tg, \A)$-module.
\end{lemma}
\subsubsection{The Weil representation as an induced module}\label{sss:The Weil representation as an induced module}
Let
\begin{equation}\label{eq:V}
{\mathbb V}=\fC[[\fx, \hbar]]((e^{\frac{c}{i\hbar}}|c\in \C))
\end{equation}
(in the notations of \eqref{eq:completion}).
This is the fiber of the bundle $\caH$ at the Lagrangian subspace $L_0=\{\fxi=0\}.$ On ${\mathbb V},$
\begin{itemize}
\item
$\A$ acts via $\fx \mapsto \fx$, $\fxi\mapsto i\hbar\frac{\partial}{\partial \fx};$
\item $\tg$ acts as above, via the identification $\tg=\oih \A;$
\item $\tP$ acts as in \ref{sss:The action of G on H}, namely: $\exp{\tg_{\geq 1}}$ acts via the exponential of the above action of $\tg$;
\begin{equation}\label{eq:Weil rep of P} \left [ \begin{array}{cc}
1&a\\
0&1\end{array}\right ] \mapsto \exp(i\hbar\frac{a\partial_{\fx}^2}{2});
\;
\left [ \begin{array}{cc}
b&0\\
0&^tb^{-1}\end{array}\right ] \mapsto T_b
\end{equation}
as in \eqref{eq:Weil rep 1}.
\end{itemize}
Define
\begin{equation}\label{eq:def ind}
\V=\tG\times_{\tP} {\mathbb V}
=\{\sum_{g\in \tG, v\in {\mathbb V}}gv\}/\sim
\end{equation}
where $g(pv)\sim (gp)v$ for $g\in \tG$ and $p\in \tP.$

Clearly, $\V$ is a $(\tG, \tg, \A)$-module, with the action defined as follows: $g_1(gv)=(g_1g)v; \; X(gv) = g(g^{-1}(X)v); a(gv) = g(g^{-1}(a)v)$ for $g_1\in \tG, \; X\in \tg, \; a\in \A.$
\begin{lemma}\label{lemma:H vs V}
$$\caH\isomoto \V$$
\end{lemma}
as $(\tG, \tg, \A)$-modules.
\section{Deformation quantization of a symplectic manifold}\label{s:Deformation quantization of a symplectic manifold}
\subsection{Deformation quantization algebra}\label{ss:Deformation quantization algebra} Recall the Fedosov construction from \cite{Fedosov}. On a symplectic manifold $(M, \omega),$ consider the bundle of algebras $\A_M$ associated to the bundle of symplectic frames on $M$ via the action of $\Sp(2n)$ on $\A$. In other words, if $g_{\alpha\beta}:U_\alpha\cap U_\beta\to \Sp(2n)$ is the cocycle defining the tangent bundle, then $\A_M$ is determined by this cocycle via the map $\Sp(2n)\to \Aut (\A)$. Similarly one defines the bundles algebras of Lie algebras $\g_M$ and $\tg_M$ and the bundles of groups $G_M$ and $\tG_M$. There exists a flat connection $\nabla$ in the bundle of algebras $\A_M$. Locally in coordinates, it is of the form
\begin{equation}\label{eq:fedosov connection}
\nabla=A_{-1}+d_{\DR}+A_0+A_1+\ldots
\end{equation}
where $A_k$ are $\g_k$-valued forms on $M$ and $A_{-1}$ is the standard form given in Darboux coordinates by
\begin{equation}\label{eq:A-1}
A_{-1}=-\frac{\partial}{\partial \fx}dx-\frac{\partial}{\partial \fxi}d\xi
\end{equation}
Locally for any $\nabla$, there is a gauge transformation $\sigma\in \Gamma(U, (G_M)_{\geq 1})$ such that $\sigma \nabla \sigma^{-1}$ is equal to
\begin{equation}\label{eq:nabla can}
\nabla_{\rm{st}}=(\frac{\partial}{\partial x}-\frac{\partial}{\partial \fx})dx+(\frac{\partial}{\partial \xi}-\frac{\partial}{\partial \fxi})d\xi
\end{equation}
Here $(G_M)_{\geq 1})=\exp((\g_M)_{\geq 1})).$ There exists a lifting of $\nabla$ to a $\tg_M$-valued connection $\tnabla.$ One has
\begin{equation}\label{eq:tnabla 2}
\tnabla^2=\theta=\oiho+\sum_{k=0}^{\infty}(i\hbar)^k \theta_k
\end{equation}
where $\theta _k$ are closed two-forms.

Given a Fedosov connection $\nabla$, the sheaf of horizontal sections of $\A_M$ is isomorphic to the sheaf $C^\infty_M[[\hbar]].$ The product on $\A_M$ induces a new product on $C^\infty_M[[\hbar]].$ This product is of the form
\begin{equation}\label{eq:def qua}
f*g=\sum_{k=1}^{\infty}fg+(i \hbar)^kP_k(f,g)
\end{equation}
where $P_k(f,g)$ are bidifferential expressions in $f$ and $g$, $1*f=f*1=f$, and $f*g-g*f=i\hbar \{f,g\}+O(\hbar^2)$. The Poisson bracket in the last formula is the symplectic Poisson bracket corresponding to $\omega$. The product as in \eqref{eq:def qua} is by definition a deformation quantization of $(M, \omega).$ Two deformation quantizations are isomorphic if there is a power series
$$T(f)=f+\sum_{k=1}^{\infty}(i \hbar)^k T_k(f)$$
where $T(f*_1 g)=T(f)*_2T(g)$ and $T_k$ are differential operators on $M$.

One proves that the set of isomorphism classes of deformation quantizations is in a bijection with the set of cohomology classes of
$$\theta\in \oih[\omega]+H^2 (M, \C[[\hbar]]).$$
The deformation quantization is constructed from $\theta$ as follows: start with a Fedosov connection $\nabla$ having a lifting $\tnabla$ with $\tnabla^2=\theta$; then pass to the sheaf of horizontal sections of $\nabla$.

To prove the above statement, one first shows that any deformation quantization is isomorphic to the sheaf of horizontal sections of some Fedosov connection. For that, start with a deformation and consider the bundle of algebras $\jets C_M^\infty[[\hbar]]$ of jets of $\C[[\hbar]]$-valued functions. This bundle carries the canonical flat connection \cite{GK} that, in coordinates, is given by \eqref{eq:nabla can}. One shows that $\jets C_M^\infty[[\hbar]]$ is a filtered bundle of algebras that is isomorphic to its associated graded bundle of algebras $\A_M$. The image of the canonical flat connection under this isomorphism is a Fedosov connection $\nabla.$ On the other hand, one shows that for any $\theta$ as above there is a Fedosov connection $\nabla$ with a $\tg$-valued lifting ${\widetilde{\nabla}}$ such that the cohomology class ${\widetilde{\nabla}}^2$ equal to $\theta$; such connections are gauge equivalent if and only if the corresponding cohomology classes $\theta$ are equal; cf. \cite{Fedosov}, \cite{Fe}, \cite{We}.

From now on, we will consider only the Fedosov connections whose class is equal to $\theta=\oih [\omega].$ All such connections are gauge equivalent, and they lead to isomorphic deformation quantizations.

\subsection{The twisted bundle $\caH_M$}\label{ss:The twisted bundle HM} Assume that $c_1(TM)=0$. Equivalently, the transition isomorphisms $g^{TM}_{\alpha\beta}$ of the tangent bundle of $M$ lift to a $\tSp(2n)$-valued \v{C}ech cocycle ${\widetilde{g}}^{TM}_{\alpha\beta}.$

Since $\tSp\subset \tG$ acts on $\caH$, we get an associated bundle of $\A_M$-modules with fiber $\caH.$ A lifting of the Fedosov connection $\nabla_{\mathbb A}$ to a $\tg$-valued connection ${\widetilde{\nabla}}$ defines a connection on $\caH.$ This connection is not flat; one has ${\widetilde{\nabla}}^2=\oih\omega$ (if the Fedosov characteristic class of our deformation is $\oih[\omega]$). To get a flat connection, we have to pass to {\em a twisted bundle}.
\begin{definition}\label{dfn:twisted bundle}
A twisted bundle is given by
$${\bf g}_{\alpha\beta}:U_{\alpha}\cap U_{\beta}\to \tG$$
subject to
$${\bf g}_{\alpha\beta} {\bf g}_{\beta\gamma}=\exp(\oih c\abc ){\bf g}_{\alpha\gamma}$$
where
$(c_{\alpha\beta\gamma})$ is a \v{C}ech two-cocycle with values in ${\mathbb R}$ representing the class of the symplectic form $\omega.$
\end{definition}
We construct the twisted bundle $\caH_M$ by defining
\begin{equation}\label{eq:Gab}
{\bf g}_{\alpha\beta}=\exp(-\oih a_{\alpha\beta}){\widetilde{g}}^{TM}_{\alpha\beta}
\end{equation}
where $(a_{\alpha\beta})$ is a \v{C}ech one-cochain with values in smooth functions whose differential is $(c_{\alpha\beta\gamma})$.

Note that, since $c\abc$ is a locally constant-valued cocycle, the notion of a connection makes sense for a twisted bundle defined by transition automorphisms subject to the identity as in Definition \ref{dfn:twisted bundle}. We define the flat connection $\nabla_{\caH}$ to be
\begin{equation}\label{eq:new connection on HM}
\nabla_{\caH}={\widetilde{\nabla}}+
\mu_\alpha
\end{equation}
where $d a_{\alpha\beta}=\mu_\alpha-\mu_\beta$.

In what follows, we give another definition of $\caH_M$ for polarized symplectic manifolds. This definition is very close to the constructions from \cite{GS}. Strictly speaking, the text until \ref{Twisted bundles and local trivializations}  is not needed for the sequel.
\subsubsection{Metalinear structures}\label{sss:Metalinear structures} Recall \cite{GS} that the metalinear group is by definition
\begin{equation}\label{eq:ML}
\ML (n,\R)=\{(g,\, \zeta)| g\in \GL(n, \R),\, \zeta^2=\det (g)\}
\end{equation}
This is a twofold cover of $\GL(n, \R)$. There is a morphism
\begin{equation}\label{eq:ML to C}
{\det} ^{\frac{1}{2}}:
\ML (n,\R)\to \C^{\times}; \; (g, \zeta)\mapsto \zeta.
\end{equation}
Denote by $\OL(n)$ the preimage of ${\rm{O}}(n)$ in $\ML(n).$ Let also
\begin{equation}
{\rm{U}}^{(2)}(n)=\{(u,\, \zeta)| u\in {\rm{U}}(n, \C),\, \zeta^2=\det (u)\}
\end{equation}
Let $\Sp^{(2)}(2n, \R)$ be the universal twofold cover of $\Sp(2n, \R).$ There is a commutative diagram
$$
\begin{CD}
\OL(n) @>>> \ML(n,\R)\\
@VVV  @VVV\\
{\rm{U}}^{(2)}(n) @>>> \Sp^{(2)}(2n, \R)$$
\end{CD}
$$
where the horizontal embeddings are homotopy equivalences.

A metalinear structure on a real vector bundle $E$ is a lifting of the transition automorphisms $g_{\alpha\beta}^E$ to an $\ML (n,\R)$-valued cocycle ${\widetilde g}_{\alpha\beta}^E$. For a real bundle $E$ with a metalinear structure, the complex line bundle $\wedge^{\frac{1}{2}}E$ is given by the transition automorphisms ${\det}^{\frac{1}{2}}({\widetilde g}_{\alpha\beta}^E)$.

Recall that a real polarization is an integrable distribution of Lagrangian subspaces. Let $\caP$ be a real polarization on $M$. In this case, automatically $2c_1(TM)=0$ (cf. \cite{S}). If $\caP$ admits a metalinear structure then $c_1(TM)=0.$

Let $P_0=\{g\in \Sp(2n)|g$ preserves the Lagrangian subspace $\{\xi=0\}\}.$ Let  $\tP_0$ be the preimage of $P_0$ in $\tSp$. From a real polarization $\caP$, one constructs local trivializations of $M$ and $TM$ such that the transition automorphisms of $TM$ are $P_0$-valued. If $\caP$ has a metalinear structure then they lift to a $\tP _0$-valued cocycle.
\subsubsection{Constructing the twisted bundle $\caH_M$ from the jet bundle}\label{sss:Constructing the twisted bundle HM from the jet bundle} From now on we assume that:
 \begin{itemize}
 \item
 $M$ is equipped with a real polarization $\caP$ that has a metalinear structure.
 \item
 There is a line bundle $\caL$ on $M$ such that $c_1(\caL)={2\pi i}[\omega]$ and the transition automorphisms of $\caL$ are $\caP$-horizontal.
\end{itemize}
 The twisted bundle $\caH_M$ is constructed as follows. For any line bundle ${\mathcal E}$ on $M$, define $\jets \Gamma (\E)$ as the bundle of jets of sections of $\E$. If the transition automorphisms of $\caE$ are horizontal with respect to $\caP$, then denote by $\jets \Gamma_{\hor} (\E)$ the bundle of jets of ${\caP}$-horizontal sections of $\caE$.
 Now, {\em formally}, write the transition automorhisms of the bundle
 \begin{equation}\label{eq:jets Gamma hor}
 \jets (\Gamma _{\hor}((\caL^{\otimes -\oih})\otimes \wedge ^{\frac{1}{2}}T_{\caP}^*))[[\hbar]]
 \end{equation}
 and apply to them the gauge transformation
 \begin{equation}\label{eq:sigma}
 \sigma _\alpha=e^{\oih \xi^\alpha\fx}
 \end{equation}
 on any Darboux chart $U_\alpha.$ Let us explain the meaning of the above formulas. First of all, $(\xi_\alpha,\,x_\alpha)$ are local Darboux coordinates on $U_\alpha$ such that leaves of $\caP$ subsets $\{x^\alpha_k=c_k, k=1,\ldots, n\};$ a choice of such a coordinate system identifies $\jets\Gamma_{\hor}(\caL^{\otimes n})\otimes \wedge^{\frac{1}{2}}T^*_{\caP}$ locally with $\C[[\fx]]$ for any integer $n$. Denote by $g_{\alpha\beta}^{\caL, n}$ the transition automorphisms with respect to these local trivializations. Put
 \begin{equation}\label{eq:transition functions of HM}
 {\bf g}_{\alpha\beta}=\sigma_\alpha g_{\alpha\beta}^{\caL,-\oih}\sigma^{-1}_\beta
 \end{equation}
 (we explain below what does it mean to substitute a formal parameter $-\oih$ for an integer $n$).
 \begin{prop}\label{prop:gerbelike} The above construction defines a cochain
 $${\bf g}_{\alpha\beta}:U_\alpha\cap U_\beta\to \tG$$
 satisfying
 $${\bf g}_{\alpha\beta}{\bf g}_{\beta\gamma}=
 \exp(\oih{c_{\alpha\beta\gamma}}){\bf g}_{\alpha\gamma}$$
 where $c_{\alpha\beta\gamma}$ is a cochain with values in $\Z$ that represents the class of the form $\omega.$
 \end{prop}
 \begin{proof} Let us explain the construction of ${\bf g}_{\alpha\beta}$ in more detail. Let $\phi _\alpha: \caL|U_\alpha\isomoto \R^{2n}\times \C$ be a local trivialization of $M$ and $\caL$. Put
 \begin{equation}\label{eq:transition and atlas 1}\phi _\alpha\phi _\beta^{-1}:\phi _\beta(U_\alpha\cap U_\beta)\times \C\isomoto \phi _\alpha(U_\alpha\cap U_\beta)\times \C;
 \end{equation}
 \begin{equation}\label{eq:transition and atlas 2}
 (x_\beta, v_\beta)\mapsto (g_{\alpha\beta}(x_\beta), h^{\caL}_{\alpha\beta}(x_\beta)v_\beta).
 \end{equation}
 Here $h^{\caL}_{\alpha\beta}(x_\beta)=g^{\caL}_{\alpha\beta}(\phi _\beta ^{-1}(x_\beta))$. Using similar notation for $T{\caP}$, define
 \begin{equation}\label{eq:transition and atlas 3}
 {\bf g}_{\alpha\beta}^{-1}(s_\alpha)(\fx)= \sigma_\alpha s_\alpha (g_{\alpha\beta}(\fx+x_\beta)-x_\alpha) h_{\alpha\beta}^{\mathcal L}(\fx+x_\beta)^{\frac{1}{i\hbar}}{\det}^{\frac{1}{2}}
 ({\widetilde{h}}^{T\caP}_{\alpha\beta})\sigma^{-1}\beta
 \end{equation}
 By definition, the power $\oih$ means the following: locally, express $h_{\alpha\beta}^{\mathcal L}$ as $\exp(f_{\alpha\beta});$ then
 $$(h_{\alpha\beta}^{\mathcal L})^{\oih}=\exp(\oih f_{\alpha\beta}).$$
 It is easy to see that \eqref{eq:transition and atlas 3} defines a map ${\bf g}_{\alpha\beta}:U_\alpha\cap U_\beta\to \tG$ which is of the form
 \begin{equation}\label{eq:transition and atlas 4}
 {\bf g}_{\alpha\beta}\in \exp(\oih f_{\alpha\beta}){\widetilde g}^{TM} _{\alpha\beta}\exp(\tg_{\geq 1});
 \end{equation}
 here ${\widetilde g}^{TM} _{\alpha\beta}$ are $\tP _0$-valued transition functions of $TM$, after the embedding $\tP_0 \to \tSp(2n).$ It is also easy to see that the \v{C}ech coboundary of $(f_{\alpha\beta})$ represents the class of $\omega$.
 \end{proof}
  The canonical connection on the jet bundle is of the form
 \begin{equation}\label{eq:connection on H}
 \nabla_{\can} = (\frac{\partial}{\partial x^\alpha}-\frac{\partial}{\partial \fx})dx^{\alpha}+\frac{\partial}{\partial \xi^\alpha}d\xi^\alpha;
 \end{equation}
 after the gauge transformation $\sigma_\alpha,$ it becomes
 \begin{equation}\label{eq:connection on H 1}
 \nabla_{\caH} = -\oih \xi^\alpha dx^\alpha +(\frac{\partial}{\partial x^\alpha}-\frac{\partial}{\partial \fx})dx^{\alpha}+(\frac{\partial}{\partial \xi^\alpha}+\oih \fx)d\xi^\alpha;
 \end{equation}
 this is a flat connection on $\caH_M.$
 \subsubsection{Constructing the deformation quantization $\A_M$ from the jet bundle}\label{sss:Constructing the deformation quantization AM from the jet bundle} Now define a bundle of algebras by means of a \v{C}ech cocycle which is the image of $G\ab$ under the projection $\tG\to G.$ This is a filtered bundle of algebras whose associated graded is isomorphic to $\A_M$ (cf., for example, \cite{Nest-Tsygan}). Abusing the notation, we will denote by $\A_M$ the above bundle of algebras itself.

 \begin{remark}\label{rmk:rees}
  Recall that the Rees ring of a ring $A$ with an increasing filtration $ F_\bullet A$ is defined by
$${\mathcal R}A=\sum \hbar ^p F_pA.$$
 Consider the bundle of algebras ${\mathcal R}\jets {\mathcal D}_{\hor}(\caL ^{\otimes-\oih}\otimes \wedge ^{\frac{1}{2}}T^*\caP)$, after the gauge transformation $\Ad \sigma_\alpha$ on any Darboux chart. This is a bundle with the fiber $C^{\infty, {\rm {pol}}}_M[\hbar]$ where the latter can be defined as follows. Let $F_0 C^\infty _M=C^\infty_{\hor}$ be the space of functions that are horizontal with respect to $\caP.$ Define inductively $F^n C^\infty _M=\{f|\{f, F_0\}\subset F_{n-1}\}.$ Put
$$C^{\infty, {\rm {pol}}}_M=\cup_n F_n C^\infty _M.$$ In local Darboux coordinates, $F_n$ consists of functions that are polynomial of degree $\leq n$ in $\xi^\alpha.$ The transition functions and the product in the above bundle of algebras extend from $C^{\infty, {\rm {pol}}}_M[\hbar]$ to $C^{\infty}_M[[\hbar]]$ to give the bundle $\A_M$ as above.
\end{remark}
One defines a connection $\nabla _{\mathbb A}$ on $\A_M$ exactly as we did for $\caH_M.$ This is a Fedosov connection. To construct its $\tg$-valued lifting, consider a bundle which is a modification of $\caH_M.$ Formally, consider the transition automorphisms of the bundle
\begin{equation}\label{eq:jets Gamma hor 1}
 \jets (\Gamma _{\hor}((\caL^{\otimes -\oih})\otimes \wedge ^{\frac{1}{2}}T_{\caP}^*))[[\hbar]]\otimes {\caL}^{\otimes\oih}
 \end{equation}
 (Rigorously speaking, these are just $G\ab exp(-\oih f\ab)$as in and after \eqref{eq:transition and atlas 3}). They have the same image in $G$ as the transition automorphisms $G\ab.$ Formally writing the connection $\nabla_{\can}\otimes 1+1\otimes \nabla^{\mathcal L}$, we get a $\tg$-valued connection $\tnabla_{\mathbb A}$ lifting $\nabla _{\mathbb A}$ and such that $\tnabla_{{\mathbb A}}^2=\oiho.$

 The above construction is a generalization of the one carried out in \cite{NT} for the case of the cotangent bundle. It is also very close to the one from \cite{BD}.

\subsection{Twisted bundles and local trivializations}\label{Twisted bundles and local trivializations}
 Another way to interpret $\caH_M$ is as follows. Let, as above, $c=(c\abc)$ be a \v{C}ech cocycle representing the class of $\omega.$
 \begin{itemize}
 \item
 For any local trivialization $T$ of $c$, a bundle $\caH_T$ of $\A_M$-modules.
 \item
 For any two local trivializations $T$ and $T'$, an isomorphism of bundles $\phi_{T,T'}:\caH _T\lisomoto \caH_{T'}.$
 \item
 For any three local trivializations $T$, $T'$, and $T''$,
 $$\phi_{T,T'}\circ \phi_{T',T''}={\bf c}_{T,T',T''}\phi_{T,T''}$$
 where ${\bf c}_{T,T',T''}\in \exp(\oih \R).$
 \item For any four trivializations,
 $${\bf c}_{T,T',T''}{\bf c}_{T,T'',T'''}={\bf c}_{T',T'',T'''}{\bf c}_{T,T',T'''}$$
\end{itemize}
Every bundle $\caH$ is a bundle of $\A_M$-modules. It is equipped with a flat connection $\nabla_{{\caH},T}$, and $\phi_{T,T'}^* \nabla_{{\caH},T}=\nabla_{{\caH},T'}.$ Also, $\phi_{T,T'}$ are isomorphisms of $\A_M$-modules.

The bundles $\caH_T$ are defined by \v{C}ech cocycles $G'\ab={G\ab}{\exp(-\oih a\ab)}$ where $a\ab$ is the cochain trivializing $c\abc$ according to the local trivialization $T$.
 \subsection{The groupoid $\tG_M$}\label{ss:The groupoid GM} We keep the assumptions from \ref{sss:Constructing the twisted bundle HM from the jet bundle}. Recall that $c$ is a chosen real-valued two-cocycle representing the cohomology class of $\omega$.

 For $x,\,y\in M$, define $\tG_{x,y}$ as the set of equivalence classes of the following data:
 \begin{enumerate}
 \item
 A smooth path $\gamma:[0,1]\to M;$ we put $\gamma(t)=x_t$ and require $x_0=x,$ $\,x_1=y.$
 \item
 A trivialization $T$ of $\gamma^*c.$
 \item
 $g_t: \caH_{T,x_0}\lisomoto \caH_{T, x_t}$ smoothly depending on $t\in [0,1]$. We require $\frac{dg_t}{dt}g_t^{-1}$ to be in the image of $\oih \C\oplus \tg_{\geq 0}$ in the space of linear operators on $\caH_{T, x}.$
 \end{enumerate}
 Two such data $(\gamma_0, g_{t,0}, T)$ and $(\gamma_1, g_{t,1}, T')$ are equivalent if there exist:
 \begin{enumerate}
 \item
 A smooth $\gamma: [0,1]^2\to M;\;$ we denote $\gamma(t,s)$ by $x_{t,s}$ and require $x_{t,0}=x$, $x_{t,1}=y.$
 \item
 A trivialization $S$ of $\gamma^*c.$
 \item
 Isomorphisms $h_{t,s}:\caH_{x_0}\lisomoto \caH_{t,s}$ smoothly depending on $(t,s)\in [0,1]^2$ and satisfying
 $$g_{t,0}=\phi ^{-1}_{S,T}(x_0)h_{t,0}\phi _{S,T}(x_{t,0});\;\;g_{t,1}=\phi ^{-1}_{S,T'}(x_0)h_{t,1}\phi _{S,T'}(x_{t,1}).$$
 \end{enumerate}
 Let $(\gamma, \, T, \,g_t)$ define an equivalence class in $\tG_{x,y}$ and $(\delta, \, S, h_t)$ a class in $\tG_{y,z}$. The composition of these classes is the class in $\tG_{x,z}$ of the data $(\gamma\circ \delta,\; R,\; k_t)$ obtained as follows: the path is the concatenation of $\gamma$ and $\delta$; the trivialization $R$ is defined in the obvious way from $T$ and $S$; the isomorphisms $k_t$ are defined from $g_t$ and $h_t$ by concatenation, using the trivialization $R$. We leave the details to the reader.

 It is easy to see that the composition is well defined. It turns the collection of $\tG_{x,y}$ into a groupoid on $M$. For any $x$ and $y$ in $M$, $\tG_{x,y}$ can be identified with $\pi _1(x,y)\times \tG$ where $\pi _1$ is the fundamental groupoid of $M$.
  \subsubsection{Connections on groupoids}\label{sss:Connections on groupoids} Assume that $\tG$ is any Lie groupoid on a manifold $M$. Denote by $\pi_1$ and $\pi_2$ the two projections $M\times M\to M,$ and by $\pi_{ij}$, $1\leq i<j\leq 3,$ the three projections $M\times M\times M\to M.$

  Let $V$ be a vector bundle on $M$. Recall that an action of the groupoid $\tG$ on $V$ is a morphism of presheaves on $M$
  $$\tG\times \pi_2^*V\to \pi^*_1 V; (\sigma, \, v)\mapsto \sigma v$$
   such that $\sigma(\tau v)=(\sigma\tau)v$ as morphisms $\pi_{12}^*\tG\times \pi_{23}^*\tG\times \pi_3^*V\to \pi^*_1 V.$

   Let $\tg$ be a bundle of Lie algebras on $M$. We assume that
  \begin{enumerate}
  \item
  $\tG$ acts on $\tg$ by Lie algebra isomorphisms;
  \item
  $\Lie(\tG_{x,x})\subset \tg_x;$
   \item
   When restricted to $\Lie(\tG_{x,x}),$ the action of $\tG$ is the usual adjoint action.
   \end{enumerate}
   Somewhat abusing the notation, we denote the action of $\tG$ on $\tg$ by $\Ad$. Note that $\tg$ can be larger than the bundle of Lie algebras $\Lie(\tG_{x,x}),$ compare to Remark \ref{rmk:Lie algebra vs Lie group}.

   A connection on $\tG$ is a morphism
of presheaves on $M\times M$
  \begin{equation}\label{eq:connection on a groupoid, 1}
  \sigma\mapsto \sigma^{-1}\nabla(\sigma)\;\;\;\tG\to\Omega^1_{M\times M}(\pi ^*_2\tg)
  \end{equation}
  satisfying
  \begin{equation}\label{eq:connection on a groupoid, 2}
  \pi_{13}^*((\sigma\tau)^{-1}\nabla(\sigma\tau))=\Ad(\tau^{-1})
  (\pi_{12}^*(\sigma^{-1}\nabla(\sigma)))+\pi^*_{23}(\tau^{-1}\nabla(\tau))
  \end{equation}
  in $\Omega^1_{M\times M\times M}(\pi_3^*(\tg)).$

   Assume now that $V$ is a vector bundle with a connection $\nabla_V$. Let $\tG$ act on $V$ together with the bundle of Lie algebras $\tg,$ and that the two actions are compatible. We say that the action is compatible with connections if
   \begin{equation}\label{eq:connection on a groupoid, 3}
   (\pi_1^*\nabla_V)(\sigma v)=\Ad(\sigma)(\sigma^{-1}\nabla(\sigma))v+\sigma\nabla_Vv
   \end{equation}
   In particular, we will often assume that $\tg$ has a flat connection $\nabla_{\tg}$ compatible with the adjoint action of $\tG$. In this case, we say that the connection on $\tG$ is {\em flat} if
   \begin{equation}\label{eq:flat connection on a groupoid}
   (\pi _2^*\nabla_{\tg})(\sigma^{-1}\nabla (\sigma))+\frac{1}{2}[\sigma^{-1}\nabla (\sigma), \sigma^{-1}\nabla (\sigma)]=0
   \end{equation}
   in $\Omega^2_{M\times M}(\pi_2^*\tg).$
   \begin{example}\label{ex:connection on the groupoid Isom}
   Let $V$ be a vector bundle with a connection $\nabla_V.$ Put $\tG_{x,y}=\Isom(V_y, V_x)$ be the set of isomorphisms $V_x\lisomoto V_y.$ Define a connection on $\tG$ that is compatible with the action of $\tG$ on $V$ as follows. In local coordinates, if $\nabla=d_{\DR}+A$ and $A_i=\pi^*_i A$ on $M\times M$, then
   \begin{equation}\label{eq:connection on the groupoid Isom}
   \sigma^{-1}\nabla(\sigma)=\sigma^{-1}d_{\DR}(\sigma)+\Ad (\sigma^{-1})A_1-A_2
   \end{equation}
   \end{example}
   \subsubsection{The connection on $\tG_M$}\label{sss:The connection on GM} We define a flat connection on $\tG_M$ to be the one induced from the above connection on $\Isom(\caH_M)$ under the morphism $\tG_M\to \Isom(\caH_M).$
\section{Oscillatory modules}\label{ss:Oscillatory modules} As above,we keep the assumptions from \ref{sss:Constructing the twisted bundle HM from the jet bundle}. We have defined:
\begin{itemize}
\item
The bundle of algebras $\A_M$ with a flat connection $\nabla_{\mathbb A}$;
\item
The bundle of Lie algebras $\tg_M$ with a flat connection $\nabla_{\tg};$
\item
The twisted bundle of modules $\caH_M$ with a flat connection $\nabla_{\caH};$
\item
The groupoid $\tG_M$ with a flat connection $\nabla,$ or $\nabla_G$.
\end{itemize}
Furthermore, $\tG_M$ acts on $\A_M$, $\tg_M$, and $\caH_M$; $\tg_M$ acts on $\A_M$ and on $\caH_M$; $\A_M$ acts on $\caH_M$; all these actions are compatible with each other (as in \ref{sss:(G,g,A)-modules}) and with the connections.
\begin{definition}\label{dfn:oscmod}
An oscillatory module is a bundle ${\mathcal V}$ with a flat connection $\nabla_{\mathcal V}$ together with the actions of $\tG_M$, $\tg_M,$ and $\A_M$. These actions are required to be compatible with one another and with the connections.
\end{definition}
The twisted bundle $\caH_M$ satisfies all the requirements but one: it is not an actual vector bundle. We will turn it into one by a cut-and-paste procedure described in the subsection below.
\subsection{Oscillatory modules and Lagrangian submanifolds}\label{ss:Oscillatory modules and Lagrangian submanifolds} Assume that $L$ is a Lagrangian submanifold of $M$. Define a vector bundle ${\mathbb V}_L$ with a flat connection $\nabla_{\mathbb V}$ as follows. Consider the restriction of $\caH_M$ to $L$. As above, choose a \v{C}ech cocycle $c$ representing the class of $\omega$. Choose a trivialization $T$ of $c$ on $L$ (this can be done because $L$ is Lagrangian). Now define $\bV_L$ as the subbundle 0f $\caH_T$ consisting of vectors whose image in the Lagrangian Grassmannian (cf. Lemma \ref{lemma:weil and lagrangian}) at every point $x$ is the tangent space $T_xL.$ It is easy to see that this subbundle is preserved by the connection $\nabla_{\caH}.$ (The induced connection is $\tg$-valued).

Let $\tP_L$ be the groupoid on $L$ constructed as follows: $({\tP_L})_{x,y}$ is the set of homotopy classes of smooth paths $\gamma: [0, 1]\to L;\;\; \gamma(t)=x_t$ together with $g_t: \bV_x\lisomoto \bV_{x_t}$ that depend smoothly in $t\in [0,1].$ We assume that $\frac{dg_t}{dt}g_t^{-1}$ is in the image of $\oih \C+\tp_{\geq 0}$ in the space of linear operators on $\bV_x$. Here $\tp$ is the Lie subalgebra of $\tg_{\geq 0}$ consisting of those $X$ whose component in $\tg_0$ lies in $\p_0\oplus \C$, $\p_0$ being the Lie algebra of $P_0.$

Now define
\begin{equation}\label{eq:VL 1}
\V_L=\tG_M\times _{\tP_L}\bV_L
\end{equation}
By definition, this means that
\begin{equation}\label{eq:VL 2}
\V_{L, x}=\{\sum _{y\in L}\tG_{x,y}\cdot\bV_{L,y}\}/\sim
\end{equation}
where
\begin{equation}\label{eq:VL 3}
gp\cdot v\sim g\cdot pv, \;\; g\in \tG_{x,y}, \; \; p\in \tP_{y,z},\;\; v\in \bV_{L,z}, \;\; y,\\; z\in L.
\end{equation}
It is clear that $\V_L$ is a $(\tG_M, \tg_M, \A_M)$-module. Indeed, the three actions are given by
\begin{equation}\label{eq:VL 4}
g(g_1\cdot v)=(gg_1)\cdot v;\;\; X(g\cdot v)=g \cdot \Ad(g^{-1})(X)v;\;\; a(g\cdot v)=g \cdot g^{-1}(a)v
\end{equation}
Now construct a flat connection on $\V$. Formally, put
\begin{equation}\label{eq:connection on V, 1}
\nabla_{\V}(g\cdot v)=g\cdot (g^{-1}\nabla_{\tG}(g))v+g\cdot \nabla_{\bV}v
\end{equation}
The meaning of this formula is as follows. First of all, $g\cdot v$ stands for the following. There is a smooth map $\varphi: U\to L$ where $U$ is an open subset of $M$, plus a section $g\in \Gamma ({\rm{graph}}(\varphi), \tG_M).$ There is also a local section $v$ of $\bV_L.$ The local section $g\cdot v$ of $\V_L$ on $U$ is defined by $x\mapsto g(x,\varphi(x))\cdot v(\varphi (x)).$ For a tangent vector $X\in T_xM,$ define
\begin{equation}\label{eq:connection on V, 2}
(\nabla_{\V})_X (g\cdot v)=g\cdot (\iota_{i_{1*}X}(g^{-1}\nabla_{\tG}(g)))v+g\cdot (\nabla_{\bV})_{\varphi _*X}(v)
\end{equation}
Here $i_1: U\to {\rm{graph}}(\varphi)$ is the embedding, and $\nabla _X v=\iota_X(\nabla v)$ is the covariant derivative along a tangent vector $X$.
 \section{The case of the two-torus}\label{s:The case of the two-torus}
 \subsection{Deformation quantization of the two-torus}\label{Deformation quantization of the two-torus} Let $M=\T^2=\R^2/\Z^2$ be the standard two-torus with coordinates $x$ and $\xi$ and with the symplectic form $\omega=d\xi dx.$
  \subsubsection{The bundles $\A_M$ and $\tg_M$}\label{sss:The bundles AM and gM on torus}
Sections of $\A_M$ are $\A$-valued functions $f(x,\xi;\fx, \fxi, \hbar)$ satisfying
\begin{equation}\label{eq:two-periodicity}
f(x+1,\xi;\fx, \fxi, \hbar)=f(x,\xi;\fx, \fxi, \hbar);\;\;f(x,\xi+1;\fx, \fxi, \hbar)=f(x,\xi;\fx, \fxi, \hbar)
\end{equation}
The Fedosov connection is of the form
\begin{equation}\label{eq:nabla a for torus}
\nabla_{\mathbb A}=(\frac{\partial}{\partial x}-\frac{\partial}{\partial \fx})dx+(\frac{\partial}{\partial \xi}-\frac{\partial}{\partial \fxi})d\xi
\end{equation}
The bundle of Lie algebras $\tg_M$ is just $\oih \A_M$ with the same connection and with the commutator given by $[a,b]=ab-ba.$
\subsubsection{The twisted bundle $\caH_M$}\label{sss:The twisted bundle HM on torus}
Consider the trivialization $T$ of the form $\omega$ on a cover $\R^2/\Z$ of $M$ with coordinates $x\in \R,\; \xi\in \R/\Z,$ given by
\begin{equation}\label{eq:triv T}
\omega = d(-xd\xi)
\end{equation}
\begin{lemma}\label{lemma:H on torus}
One can choose the data $\caL, \nabla_{\caL}$ in such a way that sections of the bundle $\caH_T$ become identified with $\caH$-valued functions on $\R^2$ satisfying
$$s(x+1, \xi)=s(x,\xi);\;\; s(x, \xi+1)=e^{\xh}s(x,\xi)$$
and the connection $\nabla_{\caH}$ becomes
$$\nabla_{\caH}=-\oih \xi dx+(\ddx-\ddfx)dx+(\ddxi+\fxh)d\xi$$
\end{lemma}
\begin{proof} Choose $\caL$ in such a way that its sections are identified with functions $s(x,\xi)$
$$s(x+1, \xi)=s(x,\xi);\;\;s(x,\xi+1)=e^{2\pi ix}s(x,\xi)$$
Then the formula
\begin{equation}\label{eq:nabla L}
\nabla_{\caL^\lambda}s(x,\xi)=(\ddx+\lambda \xi)dx+\ddxi dx
\end{equation}
defines a connection in $\caL^{\otimes \lambda}$. One has
\begin{equation}\label{eq:formulas for s 1}
\nabla^2_{\caL^\lambda}=-\lambda d\xi dx
\end{equation}
Clearly, the bundle $T_M$ is trivial, and so is $\wedge ^{\frac{1}{2}}T^*_M.$ If one writes formally the transition automorphisms of the bundle $\jets \Gamma_{\hor}(\caL^{-\otimes \oih})$, one identifies its sections with functions $s(x,\xi, \fx)$ subject to
\begin{equation}\label{eq:formulas for s 2}
s(x+1,\xi, \fx)=s(x,\xi, \fx);\;\;s(x,\xi+1, \fx)=e^{\frac{x+\fx}{i\hbar}}s(x,\xi, \fx)
\end{equation}
on which the canonical connection acts via
\begin{equation}\label{eq:formulas for s 3}
\nabla_{\can}=(\ddx-\ddfx)dx+\ddxi d\xi;
\end{equation}
After the gauge transformation
\begin{equation}\label{eq:formulas for s 4}
\sigma: s(x,\xi, \fx)\mapsto \exp(-\frac{\xi\fx}{i\hbar}),
\end{equation}
the equation \eqref{eq:formulas for s 2} becomes
\begin{equation}\label{eq:formulas for s 5}
s(x+1,\xi, \fx)=s(x,\xi, \fx);\;\;s(x,\xi+1, \fx)=e^{\frac{x}{i\hbar}}s(x,\xi, \fx)
\end{equation}
and the formula \eqref{eq:formulas for s 3} for the connection turns into the  statement of Lemma \ref{lemma:H on torus}.
\end{proof}
\subsubsection{The groupoid $G_M$}\label{sss:The groupoid GM on torus} First, let us describe some special elements of $\tG_M$. Let
\begin{equation}\label{eq:element of tG 1}
\exp(\oih((\xi_1-\xi_2+k)\fx-(x_1-x_2+l)\fxi))\in (\tG_M)_{(x_1,\xi_1),(x_2,\xi_2)}
\end{equation}
be the class of $(x_t, g_t)$ where $x_t$ is the straight line path from $(x_1,\xi_1)$ to $(x_2,\xi_2)$, followed by the straight line path from $(x_2,\xi_2)$ to $(x_2+k,\xi_2+l)$, and $g_t={\rm {id}}_{\caH}.$ Sections of $\tG_M$ are of the form
\begin{equation}\label{eq:element of tG 2}
\exp(\oih((\xi_1-\xi_2+k)\fx-(x_1-x_2+l)\fxi))g(x_1, \xi_1, x_2, \xi_2)
\end{equation}
where $g$ is a $\tG$-valued function on $\R^2\times \R^2$ satisfying
\begin{equation}\label{eq:periodicity for tG on torus 1}
\sigma=g(x_1+1, \xi_1, x_2+1, \xi_2)=g(x_1, \xi_1, x_2, \xi_2)
\end{equation}
\begin{equation}\label{eq:periodicity for tG on torus 2}
g(x_1, \xi_1+1, x_2, \xi_2+1)=\exp({\oih(x_1-x_2)})g(x_1, \xi_1, x_2, \xi_2)
\end{equation}
The connection $\nabla_{\tG}$ sends the element $\sigma$ to
\begin{equation}\label{nabla for torus}
\nabla_{\tG}(\sigma)\sigma^{-1}=\oih (\xi_2 dx_2-\xi_1 dx_1)+d_{\DR}(g)g^{-1}+
\end{equation}
$$
\Ad (g^{-1})(\oih(\fxi dx_1-\fx d\xi_1)-\oih(\fxi dx_2-\fx d\xi_2))
$$
viewed as a $\pi^*_2(\tg)$-valued one-form on $\T^2\times \T^2.$ This follows immediately from the general definition of $\nabla_{\tG}$ and from the definition of $\caH_M$ on the torus.
\section{Lagrangian submanifolds and modules}\label{ss:Lagrangian submanifolds and modules}
\subsection{The bundle of modules ${\bV_{L_n}}$}\label{ss:The bundle of modules VLn}
Let $L_n$ be the Lagrangian submanifold $
\{\xi=nx\}$ of $\T^2.$ Sections of the bundle $\bV_{L_n}$ are ${\widehat{\C}}[[\fx]]((\hbar))$-valued functions $s(x, nx+k)$, $x\in \R,\;k\in {\mathbb Z},$ satisfying the conditions as in Lemma \ref{lemma:H on torus}. The gauge transformation $s(x,nx+k)\mapsto \exp(\frac{1}{2i\hbar}nx^2+\oih kx))s(x,nx+k)$ makes this bundle trivial. After identifying $\bV_{L_n}$ with the trivial bundle, we get the following expression for the connection $\nabla_{\bV}$:
\begin{equation}\label{eq:nabla V for torus}
\nabla_{\bV}=(\ddx-\ddfx-\oih n\fx)dx
\end{equation}
The bundle of algebras ${\mathbb A}_M$ acts as follows: $\fxi$ acts by $i\hbar \ddfx+n\fx,$ and $\fx$ by multiplication by $\fx.$
\subsection{The bundle of modules $\V_{L_n}$}\label{sss:The bundle of modules cVLn} Sections of $\V_{L_n}$ are expressions
\begin{equation}\label{eq:section of cVLn}
v=\exp(\oih(\xi - nx-k)\fx)f(x,\xi)
\end{equation}
where $k$ belongs to ${\mathbb Z}$ and $f$ is a two-periodic function with values in $\caH.$ The connection $\nabla_{\V}$ transforms the function $f$ as follows:
\begin{equation}\label{eq:nabla V on the torus 1}
\nabla_{\V}: f\mapsto (-\oih(\xi-nx-k)dx+(\ddx-\ddfx-\oih n\fx)dx + (\ddxi+\oih \fx)d\xi)f,
\end{equation}
or, formally, when applied to the entire expression \eqref{eq:section of cVLn},
\begin{equation}\label{eq:nabla V on the torus 2}
\nabla_{\V}=(\ddx-\ddfx)dx+\ddxi d\xi
\end{equation}
\subsubsection{The action of $\tG_M$ on $\V_{L_n}$}\label{sss:The action of GM on VLN} The groupoid $\tG_M$ acts on $\V_{L_n}$ as follows. If
\begin{equation}\label{eq:element of tG again}
\sigma=\exp(\oih((\xi_1-\xi_2+k)\fx-(x_1-x_2+l)\fxi))g(x_1, \xi_1, x_2, \xi_2)
\end{equation}
as in \eqref{eq:element of tG 2} and
\begin{equation}\label{eq:section of cVLn again}
v=\exp(\oih(\xi_2 - nx_2-k)\fx)f(x_1,x_2, \xi_1, \xi_2)
\end{equation}
is a section of $\pi_2^*\V_{L_n}$ (cf. \eqref{eq:section of cVLn}), then
\begin{equation}\label{eq:section of cVLn again 1}
\sigma v=\exp(\oih(\xi_1 - nx_1-k)\fx)\exp(\frac{n\fx_1 ^2}{2i\hbar}-\frac{n\fx_2 ^2}{2i\hbar}+\frac{kx_1}{i\hbar}-\frac{kx_2}{i\hbar})(gf)(x_1,x_2, \xi_1, \xi_2)
\end{equation}
which is a section of $\pi_1^*\V_{L_n}$.
\subsection{Morphisms $\V_{L_m}\to \V_{L_{m+n}}$}\label{ss:Morphisms VLm to VLmn} Let $n>0$ and $a\in \frac{1}{n}{\mathbb Z}$. Consider the expression
\begin{equation}\label{eq:theta function 1}
\theta _a=\sum_{\nu \equiv a({\rm{mod}}{\mathbb Z})}e^{-\frac{n}{2i \hbar}(x+\fx+\nu)^2}
\end{equation}
which has the following meaning. Transform the above formula to get
\begin{equation}\label{eq:theta function 2}
\theta _a=\sum_{\nu \equiv a({\rm{mod}}{\mathbb Z})}
e^{-\oih n(x+\nu)\fx}e^{{-\frac{n}{2i \hbar}}\fx^2}e^{-\frac{n}{i\hbar}(x+\nu)^2}
\end{equation}
The first factor in the right hand side is viewed as an element of $\tG _{(x,\xi),(x, \xi-n(x+\nu))}$ for any point $(x,\xi);$ the second, and the third, as elements of $\tG$, and therefore of $\tG_{(x,\fxi),(x,\fxi)},$ for any $(x,\xi).$
\begin{thm} \label{thm:theta as hom} The formula $v\mapsto \theta _a v$ defines a morphism $\V_{L_m}\to \V_{L_{n+m}}$ of $(\A_M, \tG_M, \tg_M)$-modules with connections.
\end{thm}
\begin{cor}\label{cor:composition of thetas}
The subcategory of the category of oscillatory modules generated by objects $\V_{L_n}$ and morphisms $\theta_a$ is isomorphic to the subcategory of the Fukaya category of $\T^2$ generated by $L_m$ and morphisms $L_m\to L_{m+n}$ for $n>0.$
\end{cor}
\begin{proof} Note that $\theta _a$ are, formally, linear combinations of theta functions of level $n$ \cite{Theta} with $\tau=\hbar$ and $z=x+\fx,$ after an application of the Poisson summation formula and the substitution $\tau\mapsto -\frac{1}{\tau}.$ The corollary now follows from the mirror symmetry for $\T^2$, cf. \cite{PZ}.
\end{proof}


\begin{thebibliography}{10}
\bibitem{bffls}
F.~Bayen, M.~Flato, C.~Fronsdal, A.~Lichnerowicz, and
D.~Sternheimer,
\newblock {\em Deformation theory and quantization. {I}. {D}eformations of
  symplectic structures}, Ann. Physics, 111(1): 61--110, 1978.
  \bibitem{BdM} L. Boutet de Monvel,
  \newblock {\em Formal norms and star exponentials}
  \newblock Lett Math Phys {\bf 83} (2008), 213–-216.
  \bibitem{BdMG} L.~Boutet de Monvel, V.~Guillemin, {\em The spectral theory of Toeplitz operators}, Ann. of Math. Studies, {\bf 99}, Princeton University Press, 1981.
\bibitem{BD}
P.~Bressler, J.~Donin,
\newblock {\em Polarized deformation quantization}, arXiv:math/0007186.
\bibitem{bgnt1}
P.~Bressler, A.~Gorokhovsky, R.~Nest, and B.~Tsygan,
\newblock {\em Deformation quantization of gerbes},
\newblock  Adv. Math., 214(1):230--266, 2007.
\bibitem{bgnt2}
P.~Bressler, A.~Gorokhovsky, R.~Nest, and B.~Tsygan,
\newblock {\em Deformations of Azumaya algebras},
\newblock Actas del XVI Coloquio Latinoamericano de \'{A}lgebra, pages
131--152 Biblioteca de la Rev. Mat. Iberoamericana,
2007

\bibitem{bgnt3}
P.~Bressler, A.~Gorokhovsky, R.~Nest, and B.~Tsygan,
\newblock {\em Deformations of gerbes on smooth manifolds},
\newblock In: K-theory and noncommutative geometry, EMS Publishing house,
  to appear.

\bibitem{bgnt4}
P.~Bressler, A.~Gorokhovsky, R.~Nest, and B.~Tsygan,
\newblock {\em Chern character for twisted complexes},
\newblock In: Geometry and dynamics of groups and spaces. Progress in Mathematics {\bf 265}, 2008.

\bibitem{BNT} P.~Bressler, R.~Nest, R., B.~Tsygan,
\newblock \emph{Riemann-Roch theorems via deformation quantization, I, II},  Adv. Math.  \textbf{167}  (2002),  no. 1, 1--25, 26--73.
\bibitem{BS} P.~Bressler, Y.~Soibelman,
\newblock {\em Homological mirror symmetry, deformation quantization and noncommutative geometry.}
\newblock J. Math. Phys.  {\bf 45} (2004),  no. 10, 3972--3982.
\bibitem{Brylinski}
J.-L. Brylinski,
\newblock {\em Loop spaces, characteristic classes and geometric quantization},
  volume 107 of {\em Progress in Mathematics},
\newblock Birkh\"auser Boston Inc., Boston, MA, 1993.
\bibitem{C} A. Connes,
\newblock {\em Noncommutative Geometry},
\newblock Academic Press, San Diego, CA, 1994, 661 p.
\bibitem{CM} A. Connes, M. Marcolli,
\newblock  {\em A walk in the noncommutative garden}, arXiv:math/0601054
\bibitem{DAP}
A.~D'Agnolo and P.~Polesello,
\newblock Stacks of twisted modules and integral transforms,
\newblock In {\em Geometric aspects of Dwork theory. Vol. I, II},
  463--507, Walter de Gruyter GmbH \& Co. KG, Berlin, 2004.
\bibitem{Deligne}
P.~Deligne,
\newblock {\em D\'{e}formations de l'algebre des fonctions d'une vari\'{e}t\'{e} symplectique: comparaison entre Fedosov et De Wilde, Lecomte},
\newblock Selecta Math. (N.S.)  1  (1995),  no. 4, 667--697.
\bibitem{DJ} G. Dito, P. Schapira,
\newblock {\em An algebra of deformation quantization for star-exponentials on complex symplectic manifolds},
\newblock  Comm. Math. Phys.{\bf 273} (2007), no. 2, 395--414.
\bibitem{Zaslow et al 1} B. Fang, C.-C. M. Liu, D. Treumann, E. Zaslow,
\newblock {\em The Coherent-Constructible Correspondence and Homological Mirror Symmetry for Toric Varieties},
\bibitem{Zaslow et al 2} B. Fang, C.-C. M. Liu, D. Treumann, E. Zaslow,
\newblock {\em  T-Duality and Equivariant Homological Mirror Symmetry for Toric Varieties}, arXiv:0901.4276.
\bibitem{Fedosov} B.V. Fedosov,
\newblock  {\em A simple geometrical construction
of deformation quantization},
\newblock J. Diff. Geom. {\bf 40} (1994),
213--238.
\bibitem{Fe} B.~Fedosov,
\newblock {\em Deformation Quantization and Index
Theorem}, Akademie Verlag, 1994.
\bibitem{FW} E. Frenkel, E. Witten,
\newblock {\em Geometric endoscopy and mirror symmetry},
\newblock Commun. Number Theory Phys. {\bf 2} (2008),  no. 1, 113--283.
\bibitem{FG} V.~Fock, A.~Goncharov,
\newblock {\em Cluster $X$-varieties, amalgamation, and Poisson-Lie groups}, Algebraic geometry and number theory, 27--68, Progr. Math., {\bf 253}, Birkhäuser Boston, Boston, MA, 2006.
    \bibitem{Fu} K.~ Fukaya,
\newblock {\em Deformation theory, homological algebra and mirror symmetry}, Geometry and physics of branes (Como, 2001), Ser. High Energy Phys. Cosmol. Gravit., IOP, Bristol (2003), p. 121--209.
\bibitem{FOOO} K.~Fukaya, Y.-G.~Oh, H.~Ohta, K.~Ono, {\em Lagrangian intersection Floer theory--Anomaly and obstruction}, Kyoto preprint Math 00-17, 2000.
    \bibitem{F} D.~Fuks,
\newblock {\em Cohomology of infinite-dimensional Lie algebras},
\newblock Consultants Bureau, New York and London, 1986.
\bibitem{GK} I.M. Gelfand, D.A. Kazhdan.
\newblock {\em Some problems of differential geometry and
the calculation of cohomologies of Lie algebras of
vector fields},
\newblock Soviet Math. Dokl., {\bf 12},
5 (1971) 1367-1370.
\bibitem{Ge}
M.~Gerstenhaber,
\newblock On the deformation of rings and algebras.
\newblock {\em Ann. of Math. (2)}, 79:59--103, 1964.
\bibitem{Gin} V.~Ginzburg, {\em Characteristic varieties and vanishing cycles},  Invent. Math.  {\bf 84} (1986),  no. 2, p. 327-402.
\bibitem{Gi} J.~Giraud, {\em Cohomologie non ab\'{e}lienne}, Gr\"{u}ndlehren {\bf 179}, Springer Verlag (1971).
    \bibitem{Guil} V.~Guillemin, {\em Star products on compact pre-quantizable symplectic manifolds}, Lett. Math. Phys. {\bf 35} (1995), no. 1, 85--89.
    \bibitem{GS} V.~Guillemin, S.~Sternberg,
    \newblock {\em Geometric asymptotics},
    \newblock Mathematical Surveys, {\bf 14}, Amer. Math. Soc., 1977.
\bibitem{GW} S.~Gukov, E.~Witten, {\em Branes and quantization}, arXiv:hep-th/08090305.
\bibitem{H} L.~H\"{o}rmander, {\em Fourier integral operators I}, Acta
Mathematica {\bf 127}  (1971), no. 1-2, 79--183.
\bibitem{KW} A. Kapustin, E. Witten,
\newblock {\em Electric-magnetic duality and the geometric Langlands program},
\newblock  Commun. Number Theory Phys. {\bf 1} (2007),  no. 1, 1--236.
\bibitem{Kara} M.~Karasev, {\em Connections on Lagrangian varieties and some quasiclassical approximation problems},
\newblock Quantization, Coherent States, and Complex Structures, Plenum, New York (1992), p. 1053-1062.
\bibitem{Ka}
M.~Kashiwara. Quantization of contact manifolds,
\newblock {\em Publ. RIMS}, Kyoto, {\bf 32}, 1996, 1--5.
\bibitem{KS}
M.~Kashiwara and P.~Schapira.
\newblock {\em Categories and sheaves}, volume 332 of {\em Grundlehren der
  Mathematischen Wissenschaften [Fundamental Principles of Mathematical
  Sciences]}.
\newblock Springer-Verlag, Berlin, 2006.
\bibitem{KS1}
M.~Kashiwara and P.~Schapira.
\newblock Deformation quantization modules I: Finiteness and duality.
\newblock arXiv:0802.1245.
\bibitem{KS2}
M.~Kashiwara and P.~Schapira.
\newblock Deformation quantization modules II. Hochschild class.
\newblock arXiv:0809.4309.
\bibitem{K} D.~Kazhdan,
\newblock {\em Introduction to QFT},
\newblock Quantum fields and strings: a course for mathematicians, Vol. 1, 2 (Princeton, NJ, 1996/1997),  377--418, Amer. Math. Soc., Providence, RI, 1999.
\bibitem{KHMS} M.~Kontsevich,
\newblock {\em Homological algebra of mirror symmetry}, Proc. ICM Z\"{u}rich, 1994.
\bibitem{KoS} M.~Kontsevich, Y. Soibelman,
\newblock
{\em Stability structures, motivic Donaldson-Thomas invariants and cluster transformations}, arXiv:0811.2435.
\bibitem{LNT} E. Leichtnam, R. Nest, B. Tsygan,
\newblock {\em Local formula for the index of a Fourier integral operator,}
\newblock J. Differential Geom. {\bf 59} (2001),  no. 2, 269--300.
\bibitem{L} J.~Leray,
\newblock {\em Lagrangian analysis and quantum mechanics,}
\newblock Studies in applied mathematics, {\bf 7--9}, Adv. Math. Suppl. Stud., 8,
Academic Press, New York, 1983.
\bibitem{M} V.~Maslov,
\newblock {\em Th\'{e}orie de perturbations}, Dunod 1972.
\bibitem{M1} V.~Maslov,
\newblock {\em Operational methods}, Mir, Moscow, 1976.
\bibitem{MSS} A.~Mishchenko, V.~Shatalov, B.~Sternin,{\em Lagrangian manifolds and Maslov operator}, Springer Lectures in Soviet Mathematics, Springer, 1990.
\bibitem{Theta} D. Mumford,
\newblock {\em Tata Lectures on Theta, I},
\newblock Birkh\"{a}user, 1983.
\bibitem{N} D. Nadler,
\newblock {\em Microlocal branes are constructible sheaves}, arXiv:math/0612399.
       \bibitem{NZ} D.~Nadler, E.~Zaslow,
\newblock {\em  Constructible Sheaves and the Fukaya Category},
\newblock J. Amer. Math. Soc. {\bf 22} (2009), no. 1, 233--286.  arXiv:math/0604379.
\bibitem{NSS} V.~Nazaikinskii, B.-W.~Schulze, B.~Sternin,
\newblock {\em Quantization methods in differential equations}, Differential and Integral Equations and their applications, Taylor and Francis, Ltd, London, 2002.
\bibitem{NT} R.~Nest, B.~Tsygan,
\newblock {\em Remarks on Modules over deformation quantization algebras},
\newblock Moscow Math Journal  {\bf 4}  (2004),  no. 4, 911--940
\bibitem{Nest-Tsygan} R.~Nest, B.~Tsygan,
\newblock {\em Algebraic index theorem for families},
\newblock Adv. Math. {\bf 113} (1995), no. 2, 151--205.
\bibitem{P}
P.~Polesello,
 \newblock {\em Classification of deformation quantization algebroids on complex symplectic manifolds}, arXiv:math/0503400.
\bibitem{PS}
P.~Polesello, P.~Schapira,
\newblock {\em Stacks of quantization-deformation modules on complex symplectic manifolds},
Int. Math. Res. Not. {\bf 49} (2004), 2637-2664.
\bibitem{PZ} A.~Polishchuk, E.~Zaslow,
\newblock {\em Categorical mirror symmetry in the elliptic curve},
\newblock AMS/IP Stud. Adv. Math., {\bf 23}, AMS, Providence, 275--295.
\bibitem{SS}P.~Schapira, J.~P~. Shneiders,
\newblock {\em Index theorem for elliptic pairs,}
\newblock {\em Ast\'{e}risque}, {\bf 224}, 1994.
\bibitem{S} P.~Seidel,
\newblock {\em Graded Lagrangian submanifolds}, Bull. Soc. Math. France  {\bf 128}  (2000),  no. 1, p. 103-149.
\bibitem{Ta} D.~Tamarkin, {\em Microlocal condition for non-displaceability}, arXiv:math 08091584.
\bibitem {Ta2}  D.~Tamarkin,
\newblock {\em Microlocal analysis and the Fukaya category of the two-torus,} seminar talk, Northwestern University, 2007.
\bibitem {We} A.~Weinstein, 
\newblock {\em Deformation Quantization}, Seminaire Bourbaki, expos\'{e} 789, (juin 1994), Ast\'{e}risque 227, 389-409.
\end{thebibliography}
\end{document}